\documentclass[12pt]{amsart}
\usepackage{amssymb,latexsym}
\usepackage{enumerate}
\usepackage[cp1250]{inputenc}

\makeatletter
\@namedef{subjclassname@2010}{%
  \textup{2010} Mathematics Subject Classification}
\makeatother

\newtheorem{thm}{Theorem}[section]
\newtheorem{cor}[thm]{Corollary}
\newtheorem{lem}[thm]{Lemma}

\newtheorem{prop}[thm]{Proposition}

\theoremstyle{definition}
\newtheorem{defin}[thm]{Definition}
\newtheorem{rem}[thm]{Remark}
\newtheorem{exa}[thm]{Example}

\numberwithin{equation}{section}

\begin{document}

\baselineskip=17pt

\title[A note on a new ideal]{A note on a new ideal}

\author[A. Kwela]{Adam Kwela}
\thanks{The author was supported
  by the NCN (Polish National Science Centre) grant no.
  2012/07/N/ST1/03205 and by WCMCS (Warsaw Center of Mathematics and Computer Science).}
\address{Institute of Mathematics\\ University of Gda\'nsk\\Wita Stwosza 58, 80-952 Gda\'nsk,
Poland}
\email{Adam.Kwela@ug.edu.pl}

\date{}

\begin{abstract}
In this paper we study a new ideal $\mathcal{WR}$. The main result is the following: an ideal is not weakly Ramsey if and only if it is above $\mathcal{WR}$ in the Katětov order. Weak Ramseyness was introduced by Laflamme in order to characterize winning strategies in a certain game. We apply result of Natkaniec and Szuca to conclude that $\mathcal{WR}$ is critical for ideal convergence of sequences of quasi-continuous functions. We study further combinatorial properties of $\mathcal{WR}$ and weak Ramseyness. Answering a question of Filipów et al. we show that $\mathcal{WR}$ is not $2$-Ramsey, but every ideal on $\omega$ isomorphic to $\mathcal{WR}$ is Mon (every sequence of reals contains a monotone subsequence indexed by a $\mathcal{I}$-positive set). 
\end{abstract}


\maketitle

\section{Introduction}

A collection $\mathcal{I}\subset\mathcal{P}(X)$ is an \emph{ideal on $X$} if it is closed under finite unions and subsets. We additionally assume that $\mathcal{P}(X)$ is not an ideal and each ideal contains $\mathbf{Fin}=[X]^{<\omega}$. In this paper $X$ will always be a countable set. Ideal is \emph{dense} if every infinite set contains an infinite subset belonging to the ideal. The \emph{filter dual to the ideal $\mathcal{I}$} is the collection $\mathcal{I}^*=\left\{A\subset X:A^c\in\mathcal{I}\right\}$ and $\mathcal{I}^+=\left\{A\subset X:A\notin\mathcal{I}\right\}$ is the collection of all \emph{$\mathcal{I}$-positive sets}. If $Y\notin\mathcal{I}$, we can define the restriction of $\mathcal{I}$ to the set $Y$ as $\mathcal{I}\upharpoonright Y=\left\{A\cap Y:A\in\mathcal{I}\right\}$. We say that a family $\mathcal{G}$ \emph{generates the ideal $\mathcal{I}$} if $$\mathcal{I}=\left\{A:\exists_{G_0,\ldots,G_k\in\mathcal{G}}A\subset G_0\cup\ldots\cup G_k\right\}.$$
Ideals $\mathcal{I}$ and $\mathcal{J}$ are \emph{isomorphic} if there is a bijection $f:\bigcup\mathcal{J}\rightarrow\bigcup\mathcal{I}$ such that 
$$A\in\mathcal{I} \Leftrightarrow f^{-1}[A]\in\mathcal{J}.$$
For simplicity we denote $\sum (i,j)=i+j$ for $(i,j)\in\omega\times\omega$. In the entire paper $\textrm{proj}_1$ ($\textrm{proj}_2$) is the projection on the first (second) coordinate, i.e., $\textrm{proj}_i\colon\omega\times\omega\to\omega$ is given by $\textrm{proj}_i(x_1,x_2)=x_i$, for $i=1,2$.

The structure of ideals on countable sets is often described in terms of orders. We say that \emph{$\mathcal{I}$ is below $\mathcal{J}$ in the Katětov order} ($\mathcal{I}\leq_{K}\mathcal{J}$) if there is $f:\bigcup\mathcal{J}\rightarrow\bigcup\mathcal{I}$ such that
$$A\in\mathcal{I}\Rightarrow f^{-1}[A]\in\mathcal{J}.$$
If $f$ is a bijection between $\bigcup\mathcal{J}$ and $\bigcup\mathcal{I}$, we say that \emph{$\mathcal{J}$ contains an isomorphic copy of $\mathcal{I}$} ($\mathcal{I}\sqsubseteq\mathcal{J}$). Relations between $\leq_{K}$ and $\sqsubseteq$ were studied in detail in \cite{Katetov}. If $\mathcal{I}$ is a dense ideal, then $\mathcal{I}\sqsubseteq\mathcal{J}$ if and only if there is a $1-1$ function $f:\bigcup\mathcal{J}\rightarrow\bigcup\mathcal{I}$ such that $f^{-1}[A]\in\mathcal{J}$ for all $A\in\mathcal{I}$ (cf. \cite{Katetov} and \cite{Farkas}).

Ideals $\mathcal{I}$ and $\mathcal{J}$ are \emph{$\sqsubseteq$-equivalent}, if $\mathcal{I}\sqsubseteq\mathcal{J}$ and $\mathcal{J}\sqsubseteq\mathcal{I}$. Obviously, two isomorphic ideals are $\sqsubseteq$-equivalent. The converse does not hold: for instance consider
$$\mathbf{Fin}\otimes \emptyset=\{A\subseteq\omega\times\omega:\{n\in\omega:A_n\neq\emptyset\}\in\mathbf{Fin}\}$$
and
$$\mathcal{P}(\omega)\oplus\mathbf{Fin}=\{A\subseteq\{0,1\}\times\omega:\{n\in\omega:(1,n)\in A\}\in\mathbf{Fin}\}.$$
One can easily see that those ideals are $\sqsubseteq$-equivalent but not isomorphic.

In this paper we introduce a new ideal on $\omega\times\omega$.

\begin{defin}
$\mathcal{WR}$ is an ideal on $\omega\times\omega$ generated by vertical lines (which we call \emph{generators of the first type}) and sets $G$ such that for every $(i,j),(k,l)\in G$ either $i>k+l$ or $k>i+j$ (which we call \emph{generators of the second type}). Equivalently, $\mathcal{WR}$ is generated by homogeneous subsets of the coloring $\lambda\colon \left[\omega\times\omega\right]^2\to 2$ given by:
$$\lambda\left(\left\{\left(i,j\right),\left(k,l\right)\right\}\right)=\left\{\begin{array}{ll}
0 & \mbox{\boldmath{if }} k>i+j\\
1 & \mbox{\boldmath{if }} k\leq i+j\\
\end{array}\right.$$
for all $\left(i,j\right)$ below $\left(k,l\right)$ in the lexicographical order.
\end{defin}

The space $2^X$ of all functions $f:X\rightarrow 2$ is equipped with the product topology (each space $2=\left\{0,1\right\}$ carries the discrete topology). We treat $\mathcal{P}(X)$ as the space $2^X$ by identifying subsets of $X$ with their characteristic functions. All topological and descriptive notion in the context of ideals on $X$ will refer to this topology. A map $\phi:\mathcal{P}(X)\rightarrow[0,\infty]$ is a \emph{submeasure on $X$} if $\phi(\emptyset)=0$ and $\phi(A)\leq\phi(A\cup B)\leq\phi(A) + \phi(B)$, for all $A,B\subset X$. It is \emph{lower semicontinuous} if additionally $\phi(A) = \lim_{n\rightarrow\infty} \phi(A\cap \{x_0,\ldots,x_n\})$, where $X=\{x_0,x_1,\ldots\}$ is an enumeration of the set $X$. Mazur proved in \cite{Mazur} that $\mathcal{I}\in\bf{\Sigma^{0}_{2}}$ if and only if $\mathcal{I}=\mathbf{Fin}(\phi)=\left\{A\subset X:\phi(A)<\infty\right\}$ for some lower semicontinuous submeasure $\phi$.

Notice that the submeasure $\phi$ on $\omega\times\omega$ given by
$$\begin{array}{rcl}
\phi(A) & = & \inf\left\{\left|\mathcal{C}\right|:A\subset\bigcup\mathcal{C}\textrm{ and each }C\in\mathcal{C}\textrm{ is either a generator}\right.\\ 
& & \left.\textrm{of the first or of the second type of the ideal }\mathcal{WR}\right\}\end{array}$$
is lower semicontinuous and $\mathcal{WR}=\mathbf{Fin}(\phi)$. Hence $\mathcal{WR}$ is $\bf{\Sigma^{0}_{2}}$. 

We prove that $\mathcal{WR}$ is a critical ideal for weak Ramseyness. To define the latter notion we need some additional notation. If $s\in\omega^{<\omega}$, i.e., $s=\left(s(0),\ldots,s(k)\right)$ is a finite sequence of natural numbers, then by ${\rm lh}(s)$ we denote its \emph{length}, i.e., $k+1$. If $s,t\in\omega^{<\omega}$ and ${\rm lh}(s)\leq{\rm lh}(t)$, then we write $s\preceq t$ if $s(i)=t(i)$ for all $i=0,\ldots,{\rm lh}(s)-1$. We assume that $\emptyset$ is a sequence of length $0$ and $\emptyset\preceq t$ for each $t\in\omega^{<\omega}$. \emph{Concatenation of sequences $s$ and $t$} is the sequence
$$s^\frown t=(s(0),\ldots,s({\rm lh}(s)-1),t(0),\ldots,t({\rm lh}(t)-1)),$$
where $s=(s(0),\ldots,s({\rm lh}(s)-1))$ and $t=(t(0),\ldots,t({\rm lh}(t)-1))$. A set $T\subset\omega^{<\omega}$ is a \emph{tree} if for each $s\in T$ and $t\in\omega^{<\omega}$ such that $t\preceq s$, we have $t\in T$. A \emph{branch of a tree $T$} is a function $b:\omega\rightarrow\omega$ such that $\left(b(0),\ldots,b(k)\right)\in T$ for all $k\in\omega$. We sometimes identify branch $b$ with the set of all finite sequences of the form $\left(b(0),\ldots,b(k)\right)$ for $k\in\omega$ and therefore a branch can be treated as a subset of $T$. Recall also that a \emph{ramification of a tree $T\subset\omega^{<\omega}$ at $s\in T$} is the set $\left\{n\in\omega:s^\frown\left(n\right)\in T\right\}$.

\begin{defin}[cf. \cite{Laflamme}]
An ideal $\mathcal{I}$ on $X$ is \emph{weakly Ramsey} if for every tree $T\subset X^{<\omega}$ with all ramifications in $\mathcal{I}^*$ there is a $\mathcal{I}$-positive branch.
\end{defin}

We give the above definition following Laflamme. Note that the same name is used in \cite{Grigorieff} for a slightly different notion, which occurs to be equivalent to weak selectiveness (recall that an ideal $\mathcal{I}$ on $X$ is \emph{weakly selective} if every partition $(X_n)_{n\in\omega}$ of $X$ with at most one element not in $\mathcal{I}$ and such that $\bigcup_{m\geq n}X_m\notin\mathcal{I}$, for each $n\in\omega$, has a $\mathcal{I}$-positive selector). As we prove in Section $4$, weak Ramseyness and weak selectiveness do not coincide.

Weak Ramseyness was introduced by Laflamme in the context of an infinite game $G(\mathcal{I})$ in which Player I in his $n$-th move picks a set $X_n\in\mathcal{I}$ and Player II responds with $k_n\notin X_n$. Player I wins in $G(\mathcal{I})$ if $\{k_n:n\in\omega\}$ belongs to $\mathcal{I}$. Otherwise Player II wins. Laflamme proved that Player I has a winning strategy in $G(\mathcal{I})$ if and only if $\mathcal{I}$ is not weakly Ramsey. The game $G(\mathcal{I})$ was applied for instance in \cite{Hrusak} by Hru\v s\'ak in the proof of his Category Dichotomy (see also \cite{Meza}) and in \cite{KwelaSabok} for characterizing coanalytic weakly selective ideals. Recently Natkaniec and Szuca in \cite{Natkaniec} used this game in the context of ideal convergence (we discuss their result at the end of this section). 

We show that $\mathcal{WR}$ is critical for weak Ramseyness in the following sense:

\begin{thm}
\label{MainTheorem}
TFAE:
\begin{enumerate}
\item	$\mathcal{I}$ is not weakly Ramsey;
\item $\mathcal{WR}\sqsubseteq\mathcal{I}$;
\item $\mathcal{WR}\leq_{K}\mathcal{I}$.
\end{enumerate}
\end{thm}

Therefore Player I has a winning strategy in $G(\mathcal{I})$ if and only if $\mathcal{WR}\leq_{K}\mathcal{I}$ if and only if $\mathcal{WR}\sqsubseteq\mathcal{I}$. Since $\mathcal{WR}$ is dense (cf. Lemma \ref{dense}), each ideal which is not dense, has to be weakly Ramsey (this also follows from part $2$ of Proposition \ref{EquivalentConditions} and Ramsey Theorem).

To look closer at weak Ramseyness recall that $\mathcal{I}$ is \emph{selective} if for every partition $(X_n)_{n\in\omega}$ of $X$ such that $\bigcup_{m\geq n}X_m\notin\mathcal{I}$, for each $n\in\omega$, there is a $\mathcal{I}$-positive selector. Mathias in \cite{Mathias}, where instead of "selective ideal" the name "happy family" is used, proved that no analytic or coanalytic selective ideal is dense. He also showed that any ideal generated by an almost disjoint family is selective. In particular, any countably generated ideal is selective. On the other hand, Todor\u cevi\'c in \cite{Todorcevic} found an example of an analytic selective ideal which is not generated by an almost disjoint family. Zakrzewski in \cite{Zakrzewski} proved that all analytic P-ideals which are not countably generated, are not selective.

Ideal $\mathcal{I}$ is \emph{locally selective} if every partition $(X_n)_{n\in\omega}\subset\mathcal{I}$ of $X$ has a $\mathcal{I}$-positive selector. Weak selectiveness and local selectiveness were introduced in \cite{BTW} in order to generalize the notion of selective maximal ideals or ultrafilters. Later they were investigated for instance in \cite{Hrusak}, \cite{Ramsey} and \cite{Meza}.

It occurs that weak Ramseyness is between weak selectiveness and local selectiveness. Namely:

\begin{center}
selective $\implies$ weakly selective $\implies$ weakly Ramsey $\implies$ locally selective
\end{center}

For a maximal ideal all four properties coincide. Moreover, they correspond to well-known selectivity of maximal ideals or ultrafilters. However, none of the above implications can be reversed. Especially surprising may be the fact that weak Ramseyness does not coincide with local selectiveness. We discuss it in Sections $3$ and $4$. 

Local selectiveness can be characterized by an ideal $\mathcal{ED}$ on $\omega\times\omega$ generated by vertical lines and graphs of functions from $\omega$ to $\omega$, i.e., $$\mathcal{ED}=\left\{A\subset\omega\times\omega:\exists_{n,m\in\omega}\forall{k>n}\left|\left\{i\in\omega:(k,i)\in A\right\}\right|\leq m\right\}.$$
The mentioned characterization is the following: $\mathcal{I}$ is not locally selective if and only if $\mathcal{ED}\sqsubseteq\mathcal{I}$ if and only if $\mathcal{ED}\leq_{K}\mathcal{I}$ (cf. \cite{Katetov} and \cite{Meza}).

There are known other results with the same structure as Theorem \ref{MainTheorem} (in the sense that some ideal is critical for a combinatorial property through some order on ideals). Besides the one concerning local selectiveness and mentioned in Theorem \ref{1} $\mathbf{Fin}\otimes \mathbf{Fin}$ and weak P-ideals (this part of the theorem is actually straightforward) there is also a famous result of Solecki from \cite{SoleckiRB}: $\mathcal{I}$ is an analytic P-ideal which is not $\bf{\Sigma^0_2}$ if and only if it is above $\emptyset\otimes \mathbf{Fin}$ in the Rudin-Blass order, where 
$$A\in\emptyset\otimes \mathbf{Fin} \Longleftrightarrow \forall_{n\in\omega}\{m\in\omega:(n,m)\in A\}\in \mathbf{Fin}.$$

We present two applications of $\mathcal{WR}$. Following Filipów et al. (cf. \cite{Ramsey}) we say that an ideal $\mathcal{I}$ on $\omega$ is \emph{Mon} if for every sequence of reals $\left(x_n\right)_{n\in\omega}$ there is $M\notin\mathcal{I}$ with $\left(x_n\right)_{n\in M}$ monotone. An ideal $\mathcal{I}$ on $X$ is \emph{$k$-Ramsey} if for every coloring $f:[X]^2\rightarrow k$, there is $H\notin\mathcal{I}$ with $f\upharpoonright[H]^2$ constant. $\mathcal{I}$ is \emph{Ramsey} if it is $k$-Ramsey for some $k\in\omega$. It is easy to see that $2$-Ramsey implies Mon. Naturally, $\mathcal{WR}$ fails to be $2$-Ramsey as witnessed by the coloring $\lambda$. In \cite{Ramsey} authors asked about existence of a Mon ideal which is not Ramsey. Solution of this problem follows from \cite{Meza}, where Meza-Alc\'antara gave an example of a $2$-Ramsey ideal which is not $3$-Ramsey. The question about existence of a Mon ideal which is not $2$-Ramsey was still open. We show that $\mathcal{WR}$ is such an ideal. More precisely, we prove the following theorem.

\begin{thm}
\label{Mon}
Every ideal on $\omega$ isomorphic to $\mathcal{WR}$ is Mon.
\end{thm}

The second application of the ideal $\mathcal{WR}$ is connected with ideal convergence. Let $\mathcal{I}$ be an ideal on $\omega$. A sequence $(x_i)_{i\in \omega}$ of reals is \emph{$\mathcal{I}$-convergent} to $x\in \mathbb{R}$ if 
$$\left\{i\in \omega:|x_i-x|\geq\epsilon\right\}\in\mathcal{I}$$
for every $\epsilon>0$. A function $f\colon X\to\mathbb{R}$ is a \emph{pointwise $\mathcal{I}$-limit of a sequence of functions $(f_i)_{i\in \omega}$} if $(f_i(x))_{i\in \omega}$ is $\mathcal{I}$-convergent to $f(x)$ for every $x\in X$. 

For a family $\mathcal{F}$ of real-valued functions by $LIM(\mathcal{F})$ we denote the family of all functions which can be represented as a pointwise limit of a sequence of functions from $\mathcal{F}$. For instance, if $\mathcal{C}(X)$ denotes the family of continuous functions defined on a topological space $X$, then $LIM(\mathcal{C}(X))$ is the first Baire class. Similarly, by $\mathcal{I}$-$LIM(\mathcal{F})$ we denote the family of all functions which can be represented as a pointwise $\mathcal{I}$-limit of a sequence of functions from $\mathcal{F}$. Note that actually the notion of $\mathcal{I}$-convergence makes sense for all ideals on countable sets (not necessarily on $\omega$) if one considers $\{x_i:i\in\bigcup\mathcal{I}\}\subset\mathbb{R}$ instead of a sequence of reals.

Let $X$ be a topological space. A function $f\colon X\to\mathbb{R}$ is \emph{quasi-continuous} if for all $\epsilon>0$, $x_0\in X$ and open neighborhood $U\ni x_0$ there is nonempty open $V\subset U$ such that $|f(x)-f(x_0)|<\epsilon$ for all $x\in V$. By $\mathcal{QC}(X)$ we denote the family of all real-valued quasi-continuous functions defined on the space $X$. All continuous functions as well as all left-continuous (right-continuous) functions are quasi-continuous. In \cite{Grande} Grande proved that if $X$ is a metrizable Baire space then $LIM(\mathcal{QC}(X))$ is equal to the family of \emph{pointwise discontinuous functions} defined on $X$, i.e., functions with dense sets of continuity points. 

Ideal convergence has a long history, going back to Cartan's paper from the thirties (cf. \cite{Cartan}) as well as Grimeisen and Katětov papers from the sixties (cf. \cite{Grimeisen}, \cite{Kat1} and \cite{Kat2}). Later many papers were published in this area including \cite{Marciszewski}, \cite{Wilczynski} and \cite{SoleckiRank}. During the last several years, this problem appeared in numerous publications such as \cite{Balcerzak}, \cite{Debs}, \cite{Filipow}, \cite{FSz}, \cite{Kwela} and \cite{Reclaw}. Recently Natkaniec and Szuca in \cite{Natkaniec} obtained a result for the family of quasi-continuous functions. They used the game $G(\mathcal{I})$.

\begin{thm}[cf. \cite{Natkaniec}]
\label{QC}
Let $\mathcal{I}$ be a Borel ideal. TFAE:
\begin{enumerate}
\item	$\mathcal{I}$-$LIM(\mathcal{QC}(X))=LIM(\mathcal{QC}(X))$ for every metrizable Baire space $X$;
\item $\mathcal{I}$ is weakly Ramsey.
\end{enumerate}
\end{thm}

A similar result for sequences of continuous functions was obtained by Laczkovich and Recław in \cite{Reclaw}. They used a slightly different infinite game, which was also considered by Laflamme in \cite{Laflamme}. Later the same method was applied in \cite{FSz} and \cite{Kwela}.  

\begin{thm}[cf. \cite{Reclaw}]
\label{1}
Let $\mathcal{I}$ be a Borel ideal. TFAE:
\begin{enumerate}
\item	$\mathcal{I}$-$LIM(\mathcal{C}(X))=LIM(\mathcal{C}(X))$ for every uncountable Polish space $X$;
\item $\mathcal{I}$ is a weak P-ideal;
\item $\mathbf{Fin}\otimes \mathbf{Fin}\not\leq_K\mathcal{I}$;
\item $\mathbf{Fin}\otimes \mathbf{Fin}\not\sqsubseteq\mathcal{I}$.
\end{enumerate}
\end{thm}

Recall that $\mathbf{Fin}\otimes \mathbf{Fin}$ is the ideal on $\omega\times\omega$ given by
$$A\in \mathbf{Fin}\otimes \mathbf{Fin} \Longleftrightarrow \{n\in\omega:\{m\in\omega:(n,m)\in A\}\notin \mathbf{Fin}\}\in \mathbf{Fin}$$
and $\mathcal{I}$ is a \emph{weak P-ideal} if for every $(X_i)_{i\in \omega}\subset\mathcal{I}$ there is $X\notin\mathcal{I}$ with $X\cap X_i$ finite for all $i$.

In the above $\mathbf{Fin}\otimes \mathbf{Fin}$ is critical for ideal limits of continuous functions. It was unknown if Theorem \ref{QC} has a counterpart of $\mathbf{Fin}\otimes \mathbf{Fin}$. By our characterization of weak Ramseyness we solve this problem:

\begin{cor}
Let $\mathcal{I}$ be a Borel ideal. TFAE:
\begin{enumerate}
\item	$\mathcal{I}$-$LIM(\mathcal{QC}(X))=LIM(\mathcal{QC}(X))$ for every metrizable Baire space $X$;
\item	$\mathcal{I}$ is weakly Ramsey;
\item $\mathcal{WR}\not\leq_K\mathcal{I}$;
\item $\mathcal{WR}\not\sqsubseteq\mathcal{I}$.
\end{enumerate}
\end{cor}

This paper is organized as follows. Theorem \ref{Mon} is proved in Section $2$. Section $3$ is devoted to the study of weak Ramseyness. In Section $4$ we compare weak Ramseyness with other selective properties of ideals and show that there are at least two nonisomorphic locally selective ideals which are not weakly Ramsey. Section $5$ contains the proof of Theorem \ref{MainTheorem}.

\section{The Mon property}

\begin{rem}
\label{uwaga}
If $a_0,\ldots,a_k\in\omega\times\omega$ are such points that
\begin{itemize} 
	\item $\textrm{proj}_1\left(a_i\right)<\textrm{proj}_1\left(a_j\right)$ for all $i<j\leq k$,
	\item $\sum a_i>\textrm{proj}_1\left(a_k\right)$ for all $i\leq k$,
\end{itemize}
then $a_0,\ldots,a_k$ cannot be covered by $k$ many generators of the second type of the ideal $\mathcal{WR}$. Indeed, otherwise two points out of $a_0,\ldots,a_k$ would be covered by one of those generators. Assume that $a_i$ and $a_j$, for some $i<j\leq k$, are covered by one generator of the second type of the ideal $\mathcal{WR}$. Then we should have $\textrm{proj}_1\left(a_j\right)>\sum a_i$, however $\sum a_i>\textrm{proj}_1\left(a_k\right)>\textrm{proj}_1\left(a_j\right)$. A contradiction.
\end{rem}

\begin{proof}[Proof of Theorem \ref{Mon}]
Set any bijection $\pi:\omega\times\omega\rightarrow\omega$ and sequence of reals $\left(x_n\right)_{n\in\omega}$. Consider the sequence $\left(y_{a}=x_{\pi(a)}\right)_{a\in\omega^2}$ and define
$$T=\left\{i\in\omega:\left(y_{(i,j)}\right)_{j\in\omega}\textrm{ contains a nondecreasing subsequence}\left(y_{(i,j^i_k)}\right)_{k\in\omega}\right\},$$
$$T'=\left\{i\in\omega:\left(y_{(i,j)}\right)_{j\in\omega}\textrm{ contains a nonincreasing subsequence}\left(y_{(i,j^i_k)}\right)_{k\in\omega}\right\}.$$
Since every sequence contains a monotone subsequence, one of those sets must be infinite. Suppose that $T=\left\{t_0<t_1<\ldots\right\}$ is infinite (the other case is similar). Consider now the sequence $\left(\lim_{k\rightarrow\infty}y_{(t_l,j^{t_l}_k)}\right)_{l\in\omega}\subset\mathbb{R}\cup\left\{\infty\right\}$. It contains some monotone subsequence 
$$\left(\lim_{k\rightarrow\infty}y_{(t_{l_s},j^{t_{l_s}}_k)}\right)_{s\in\omega}$$
which is either nondecreasing or decreasing. Define
$$y(s,k)=y_{(t_{l_s},j^{t_{l_s}}_k)}$$
and
$$Y=\left\{(t_{l_s},j^{t_{l_s}}_k):k,s\in\omega\right\}.$$
There are four possible cases.\\
{\bf Case 1.} If $\left(\lim_{k\rightarrow\infty}y(s,k)\right)_{s\in\omega}$ is increasing, then we construct inductively points $a_i\in Y$, $i\in\omega$, as follows. Let $a_0$ be any point in $Y\cap\left\{t_{l_0}\right\}\times\omega$. Suppose that $a_j$, for $j<i$, are constructed. Let $a_i$ be any point in $Y\cap\left\{t_{l_i}\right\}\times\omega$ such that
	\begin{itemize}
		\item $\pi(a_i)>\pi(a_{i-1})$;
		\item $y_{a_i}> y_{a_{i-1}}$;
		\item $\sum a_i>t_{l_{2i}}$.
	\end{itemize}
The three imposed in each step conditions eliminate only finitely many points from $Y\cap(\left\{t_{l_{i}}\right\}\times\omega)$ since $\lim_{k\rightarrow\infty}y(s,k)<\lim_{k\rightarrow\infty}y(s',k)$ for $s<s'$. Hence, it is always possible to choose such points.\\
Define $M=\left\{a_i:i\in\omega\right\}$. Then the sequence $(y_{a_i})_{i\in\omega}$ is increasing. Moreover, $M$ does not belong to $\mathcal{WR}$. Indeed, otherwise there would be $k$ such that $M$ would be covered by $k$ generators of the second type of the ideal $\mathcal{WR}$. However it is impossible by Remark \ref{uwaga} applied to points $a_p,\ldots,a_{p+k+1}$, where $p=1+2+\ldots +k$, since $\sum a_{p+i}>t_{l_{2p}}>t_{l_{p+k+1}}$ for $i\leq k+1$.\\
{\bf Case 2.} If the sequence $\left(\lim_{k\rightarrow\infty}y(s,k)\right)_{s\in\omega}$ is constant and for infinitely many $s$ the sequence $\left(y(s,k)\right)_{k\in\omega}$ is constant from some point on, then let $\left(s_r\right)_{r\in\omega}$ be the sequence of those $s$. We construct inductively points $a_i\in Y$, $i\in\omega$, similarly as in Case $1$, assuring that each $a_i$ is in $Y\cap\left\{t_{l_{s_i}}\right\}\times\omega$ and such that
	\begin{itemize}
		\item $\pi(a_i)>\pi(a_{i-1})$;
		\item $y_{a_i}=y_{a_{i-1}}$;
		\item $\sum a_i>t_{l_{2i}}$.
	\end{itemize}
Then $M=\left\{a_i:i\in\omega\right\}$ does not belong to $\mathcal{WR}$ by Remark \ref{uwaga} (for the same reasons as in Case $1$) and the sequence $(y_{a_i})_{i\in\omega}$ is constant.\\
{\bf Case 3.} If the sequence $\left(\lim_{k\rightarrow\infty}y(s,k)\right)_{s\in\omega}$ is constant and for infinitely many $s$ the sequence $\left(y(s,k)\right)_{k\in\omega}$ is increasing, then let $\left(s_r\right)_{r\in\omega}$ be the sequence of those $s$. We construct inductively points $a_i\in Y$, $i\in\omega$, similarly as in Case $1$, assuring that each $a_i$ is in $Y\cap\left\{t_{l_{s_i}}\right\}\times\omega$ and such that
	\begin{itemize}
		\item $\pi(a_i)>\pi(a_{i-1})$;
		\item $y_{a_i}> y_{a_{i-1}}$;
		\item $\sum a_i>t_{l_{2i}}$.
	\end{itemize}
Then $M=\left\{a_i:i\in\omega\right\}$ does not belong to $\mathcal{WR}$ by Remark \ref{uwaga} (for the same reasons as in Case $1$) and the sequence $(y_{a_i})_{i\in\omega}$ is increasing.\\
{\bf Case 4.} If the sequence$\left(\lim_{k\rightarrow\infty}y(s,k)\right)_{s\in\omega}$ is decreasing, then we construct inductively points $a_i\in Y$, $i\in\omega$, as follows. Let $a_0$ be any point in $Y\cap(\left\{t_{l_0}\right\}\times\omega)$ such that $y_{a_0}\geq\lim_{k\rightarrow\infty}y(1,k)$. It is possible to choose such point since $\lim_{k\rightarrow\infty}y(0,k)>\lim_{k\rightarrow\infty}y(1,k)$. Suppose that $a_j$, for $j<i$, are constructed and let $a_i$ be any point in $Y\cap(\left\{t_{l_{i}}\right\}\times\omega)$ satisfying
	\begin{itemize}
		\item $\pi(a_i)>\pi(a_{i-1})$;
		\item $y_{a_i}> \lim_{k\rightarrow\infty}y(i+1,k)$;
		\item $\sum a_i>t_{l_{2i}}$.
	\end{itemize}
Notice that $y(i+1,k)\in Y\cap (\{t_{l_{i+1}}\}\times\omega)$ and $\lim_{k\rightarrow\infty}y(i,k)>\lim_{k\rightarrow\infty}y(i+1,k)$, so it is possible to choose points satisfying the second condition. Again the three imposed conditions eliminate only finitely many points from $Y\cap(\left\{t_{l_{i}}\right\}\times\omega)$. Hence, it is always possible to choose such points.\\
Define $M=\left\{a_i:i\in\omega\right\}$. Then, again, $M$ is not in $\mathcal{WR}$ by Remark \ref{uwaga} (for the same reasons as in Case $1$) and the sequence $(y_{a_i})_{i\in\omega}$ is decreasing.
\end{proof}

\section{Weak Ramseyness}

Let us recall two results of Grigorieff.

\begin{prop}[cf. \cite{Grigorieff}]
\label{s.}
Let $\mathcal{I}$ be an ideal on $\omega$. TFAE:
\begin{enumerate}
\item	$\mathcal{I}$ is selective.
\item	For every tree $T\subset\omega^{<\omega}$, such that no finite intersection of its ramifications is in the ideal $\mathcal{I}$, there is an $\mathcal{I}$-positive branch.
\item	For every decreasing sequence $(X_n)_{n\in\omega}$ of $\mathcal{I}$-positive subsets of $\omega$ there exists an increasing function $f:\omega\rightarrow\omega$, with $\mathcal{I}$-positive range and such that $f(n+1)\in X_{f(n)}$ for each $n\in\omega$.
\end{enumerate}
\end{prop}

\begin{prop}[cf. \cite{Grigorieff}]
\label{w.s.}
Let $\mathcal{I}$ be an ideal on $\omega$. TFAE:
\begin{enumerate}
\item	$\mathcal{I}$ is weakly selective.
\item For every $\mathcal{I}$-positive $Y$ and every tree $T\subset Y^{<\omega}$ whose ramifications are in $\left(\mathcal{I}\upharpoonright Y\right)^*$, there is an $\mathcal{I}$-positive branch.
\item For every $\mathcal{I}$-positive $Y$ and every coloring $f:[Y]^2\rightarrow 2$, such that for each $x\in Y$ either $\left\{y\in Y:f\left(\left\{x,y\right\}\right)=0\right\}$ is in the ideal or $\left\{y\in Y:f\left(\left\{x,y\right\}\right)=1\right\}$ is in the ideal, there is $\mathcal{I}$-positive subset $H$ of $Y$ with $f\upharpoonright[H]^2$ constant.
\item For every decreasing sequence $(X_n)_{n\in\omega}$ of $\mathcal{I}$-positive subsets of $\omega$ and such that $X_n\setminus X_{n+1}\in\mathcal{I}$ for each $n$, there exists an increasing function $f:\omega\rightarrow\omega$, with $\mathcal{I}$-positive range and such that $f(n+1)\in X_{f(n)}$ for each $n\in\omega$.
\end{enumerate}
\end{prop}

The next Proposition \ref{EquivalentConditions} is a weak Ramseyness counterpart of Propositions \ref{s.} and \ref{w.s.}. Notice that the condition corresponding to $(1)$ of Proposition \ref{w.s.} is the last one. However it is only a restatement of the third condition. It may be quite surprising that local selectiveness, which seems to be a natural counterpart of weak selectiveness for weak Ramseyness, does not imply weak Ramseyness (for detailed arguments see Section $4$). Grigorieff's proof does not work in this case.

\begin{prop}[Essentially Grigorieff]
\label{EquivalentConditions}
Let $\mathcal{I}$ be an ideal on $\omega$. TFAE:
\begin{enumerate}
\item $\mathcal{I}$ is weakly Ramsey.
\item For every coloring $f:[\omega]^2\rightarrow 2$, such that for each $x\in \omega$ either $\left\{y\in \omega:f\left(\left\{x,y\right\}\right)=0\right\}\in\mathcal{I}$ or $\left\{y\in \omega:f\left(\left\{x,y\right\}\right)=1\right\}\in\mathcal{I}$, there is $\mathcal{I}$-positive $H$ with $f\upharpoonright[H]^2$ constant.
\item For every decreasing sequence $(X_n)_{n\in\omega}\subset\mathcal{I}^*$ of subsets of $\omega$, there exists an increasing function $f:\omega\rightarrow\omega$, with $\mathcal{I}$-positive range and such that $f(n+1)\in X_{f(n)}$ for each $n\in\omega$.
\item For every partition $(X_n)_{n\in\omega}\subset\mathcal{I}$ of $\omega$, there exists an increasing function $f:\omega\rightarrow\omega$, with $\mathcal{I}$-positive range and such that $f(n+1)\in\bigcup_{i>f(n)}X_i$ for each $n\in\omega$.
\end{enumerate}
\end{prop}

Formally the last conditions of both Propositions \ref{s.} and \ref{w.s.} as well as two last conditions of Proposition \ref{EquivalentConditions} make sense only for ideals on $\omega$. Therefore, the assumption that $\mathcal{I}$ is an ideal on $\omega$ is required. Although, it should be pointed out that selectivity, weak selectivity and weak Ramseyness are invariant under isomorphism of ideals. Hence, one can apply those conditions for any ideals on countable sets when considering any isomorphic copy on $\omega$ of the given ideal.

The proof is just an adaptation of Grigorieff's proofs for the case of weak selectiveness. Although we attach it here for completeness.

\begin{proof}
{\bf(3) $\Leftrightarrow$ (4):} Obvious.\\
{\bf(1) $\Rightarrow$ (2):} Assume that $\mathcal{I}$ is weakly Ramsey. Let $f:[\omega]^2\rightarrow 2$ be a coloring, such that for each $x\in \omega$ either the set $C^0_x=\left\{y\in \omega:f\left(\left\{x,y\right\}\right)=0\right\}$ is in the ideal or the set $C^1_x=\left\{y\in \omega:f\left(\left\{x,y\right\}\right)=1\right\}$ is in the ideal. Define inductively a tree $T\subset\omega^{<\omega}$ as follows:
\begin{itemize}
	\item the ramification $A_\emptyset$ of $T$ at $\emptyset$ is $\omega$.
	\item if $s=\left(s(0),\ldots,s(k)\right)\in T$, then the ramification $A_s$ of $T$ at $s$ is the intersection of $A_{\left(s(0),\ldots,s(k-1)\right)}$ with that of the sets $C^0_{s(k)}$ and $C^1_{s(k)}$ which is in $\mathcal{I}^*$.
\end{itemize}
Notice that all ramifications of $T$ are in $\mathcal{I}^*$, so there is a branch $b$ not in $\mathcal{I}$. For each $n$ there is $i(n)\in 2$ such that 
$$A_{\left(b(0),\ldots,b(n+m)\right)}\subset A_{\left(b(0),\ldots,b(n)\right)}\subset C^{i(n)}_{b(n)}$$ 
for each $m$ and therefore $f\left(\left\{b(n),b(n+m)\right\}\right)=i(n)$ for all $m$. Since $b\notin\mathcal{I}$, one of the sets $\left\{b(n):i(n)=0\right\}$ and $\left\{b(n):i(n)=1\right\}$ is not in the ideal. Hence, it is the required set.\\
{\bf(2) $\Rightarrow$ (3):} Assume that $\mathcal{I}$ satisfies condition $(2)$ and let $(X_n)_{n\in\omega}\subset\mathcal{I}^*$ be a decreasing sequence of subsets of $\omega$. Define a coloring $f:[\omega]^2\rightarrow 2$ as follows: $$f\left(\left\{n,m\right\}\right)=0\Leftrightarrow m\in X_n,$$
for $n<m$. Then $\left\{m\in \omega:f\left(\left\{n,m\right\}\right)=0\right\}\supset X_n\setminus (n+1)\in\mathcal{I}^*$ for each $n$, so there is $H\notin\mathcal{I}$ such that $f\upharpoonright[H]^2=0$. Let $h\colon\omega\to H$ be an increasing enumeration of the set $H$. Then $h[\omega]=H$ is not in the ideal and $h(n+1)\in\left\{m\in \omega:f\left(\left\{h(n),m\right\}\right)=0\right\}\subset X_{h(n)}$. Hence, $h$ is the required function.\\
{\bf(3) $\Rightarrow$ (1):} Assume that $\mathcal{I}$ satisfies condition $(3)$.\\
Firstly we will show that for every family $\left\{X_s\right\}_{s\in\omega^{<\omega}}\subset\mathcal{I}^*$, there is an increasing function $h:\omega\rightarrow\omega$, with range not in $\mathcal{I}$ and such that $h(n)\in X_{\left(h(0),\ldots,h(n-1)\right)}$ for each $n\in\omega$. Indeed, let $\left\{X_s\right\}_{s\in\omega^{<\omega}}\subset\mathcal{I}^*$ and notice that without lost of generality we can assume that this family has the following property: if $s,t\in\omega^{<\omega}$ are such that ${\rm lh}(s)\leq {\rm lh}(t)$ and $\max_{k<{\rm lh}(s)}s(k)\leq\max_{k<{\rm lh}(t)}t(k)$, then $X_t\subset X_s$. Indeed, we can let $X'_t$ be the intersection of $X_t$ with all (finitely many) $X_s$ such that ${\rm lh}(s)\leq {\rm lh}(t)$ and $\max_{k<{\rm lh}(s)}s(k)\leq\max_{k<{\rm lh}(t)}t(k)$. Then $X'_s\in\mathcal{I}^*$ and $X'_s\subset X_s$, so the desired function for the family $\left\{X'_s\right\}_{s\in\omega^{<\omega}}$ is also good for $\left\{X_s\right\}_{s\in\omega^{<\omega}}$. Let $s_n$ be the constant sequence of value $n$ and length $n+1$. Then $(X_{s_n})_{n\in\omega}$ is a decreasing family of sets in $\mathcal{I}^*$. By the assumption there is an increasing $h:\omega\rightarrow\omega$, with range not in $\mathcal{I}$ and such that $h(n+1)\in X_{s_{h(n)}}$ for each $n\in\omega$. The sequence $\left(h(0),\ldots,h(n)\right)$ is of length $n+1$ and its maximum is $h(n)$. On the other hand, the sequence $s_{h(n)}$ is of length $h(n)+1\geq n+1$ (since $h$ is increasing) and its maximum is also $h(n)$. Hence $X_{s_{h(n)}}\subset X_{\left(h(0),\ldots,h(n)\right)}$ and $h(n+1)\in X_{\left(h(0),\ldots,h(n)\right)}$.\\
Let now $T\subset\omega^{<\omega}$ be a tree with ramifications in $\mathcal{I}^*$. Define $(X_s)_{s\in\omega^{<\omega}}$ as follows: if $s\in T$ then $X_s$ is the ramification of $T$ at $s$ and otherwise $X_s=\omega$. Then there is an increasing function $h:\omega\rightarrow\omega$, with range not in $\mathcal{I}$ and such that $h(n)\in X_{\left(h(0),\ldots,h(n-1)\right)}$ for each $n\in\omega$. To finish the proof, we will show inductively that $h$ is a branch of $T$. $h(0)\in X_\emptyset$ and $\emptyset$ is in $T$, so $\left(h(0)\right)\in T$. Suppose now that $\left(h(0),\ldots,h(k-1)\right)\in T$. Then
$$h(k)\in X_{\left(h(0),\ldots,h(k-1)\right)}=\left\{n\in\omega:\left(h(0),\ldots,h(k-1),n\right)\in T\right\},$$
therefore $\left(h(0),\ldots,h(k-1),h(k)\right)\in T$.
\end{proof}

\begin{cor}
\label{implikacje}
The following implications hold:
\begin{enumerate}
\item All weakly selective ideals are weakly Ramsey.
\item All weakly Ramsey ideals are locally selective.
\end{enumerate}
\end{cor}

\begin{cor}
\label{w.R.}
The ideal $\mathcal{WR}$ is not weakly Ramsey. 
\end{cor}

\begin{proof}
$\mathcal{WR}$ is generated by homogeneous subsets of $\lambda\colon \left[\omega\times\omega\right]^2\to 2$ given by:
$$\lambda\left(\left(i,j\right),\left(k,l\right)\right)=\left\{\begin{array}{ll}
0 & \mbox{\boldmath{if }} k>i+j\\
1 & \mbox{\boldmath{if }} k\leq i+j\\
\end{array}\right.$$
for all $\left(i,j\right)$ below $\left(k,l\right)$ in the lexicographical order. It suffices to observe that
$$\left\{b\in \omega\times\omega:\lambda\left(\left\{a,b\right\}\right)=1\right\}\in\mathcal{WR},$$
for each $a\in\omega\times\omega$, since this set is covered by finitely many vertical lines.
\end{proof}

\section{Comparison of weak Ramseyness with weak selectiveness and local selectiveness}

Both implications in Corollary \ref{implikacje} cannot be reversed. In the case of the second one it may be quite surprising, because a natural counterpart of first condition of Grigorieff's Proposition \ref{w.s.} for weak Ramseyness would be local selectiveness.

The following discussion will lead us to a conclusion that there are weakly Ramsey ideals which are not weakly selective.

\begin{prop}
Let $\mathcal{I}$ be an ideal on $X$ and $A\notin\mathcal{I}$.
\begin{enumerate}
\item If the ideal $\mathcal{I}\upharpoonright A$ is weakly Ramsey, then $\mathcal{I}$ is weakly Ramsey.
\item If $\mathcal{I}$ is weakly selective, then the ideal $\mathcal{I}\upharpoonright A$ is weakly selective.
\end{enumerate}
\end{prop}

\begin{proof}
{\bf(1):} Assume first that $\mathcal{I}\upharpoonright A$ is weakly Ramsey. We will show that $\mathcal{I}$ is weakly Ramsey. Let $T\subset X^{<\omega}$ be a tree with all ramifications $X_s$ in $\mathcal{I}^*$. Define a tree $T'\subset A^{<\omega}$ in the following way: if $s\in T'$, then the ramification $B_s$ of $T'$ at $s$ is $X_s\cap A$. Observe that $B_s\in(\mathcal{I}\upharpoonright A)^*$ for all $s\in T'$, so there is a $\mathcal{I}\upharpoonright A$-positive branch $b$ of $T'$. To conclude the proof, notice that $b$ is a branch of $T$, which is not in $\mathcal{I}$.\\
{\bf(2):} Assume now that $\mathcal{I}$ is weakly selective. We will show that $\mathcal{I}\upharpoonright A$ is weakly selective. Let $(X_n)_{n\in\omega}$ be a partition of the set $A$ such that $X_i\in\mathcal{I}\upharpoonright A$ for $i>0$. We define a partition of the set $X$ in the following way: $Y_0=(X\setminus A)\cup X_0$ and $Y_n=X_n$ for $n>0$. Then $Y_n\in\mathcal{I}$ for all $n>0$. Moreover, $\bigcup_{m\geq n}Y_m\notin\mathcal{I}$ for all $n\in\omega$. Hence, there is a $\mathcal{I}$-positive selector $S$ of the partition $(Y_n)_{n\in\omega}$. Observe that $S\cap A$ is a $(\mathcal{I}\upharpoonright A)$-positive selector of the partition $(X_n)_{n\in\omega}$.
\end{proof}

By $X\oplus Y$ we denote the disjoint union of sets $X$ and $Y$, i.e., $X\oplus Y=(\left\{0\right\}\times X)\cup(\left\{1\right\}\times Y)$. If $\mathcal{I}$ and $\mathcal{J}$ are ideals on $X$ and $Y$, respectively, then $\mathcal{I}\oplus\mathcal{J}$ given by 
$$A\in\mathcal{I}\oplus\mathcal{J}\Leftrightarrow\left\{x\in X: (0,x)\in A\right\}\in\mathcal{I}\wedge\left\{y\in Y: (1,y)\in A\right\}\in\mathcal{J}$$
is an ideal on $X\oplus Y$.

The existence of a weakly Ramsey ideal which is not weakly selective follows from the above Proposition. For instance, let $\emptyset\otimes \mathbf{Fin}$ be the ideal on $\omega\times\omega$ consisting of those sets which are finite on every vertical line. Observe that $\emptyset\otimes \mathbf{Fin}\upharpoonright\left\{0\right\}\times\omega=\mathbf{Fin}$, hence 
$$\emptyset\otimes \mathbf{Fin}=\left(\emptyset\otimes \mathbf{Fin}\upharpoonright\left\{0\right\}\times\omega\right)\oplus\left(\emptyset\otimes \mathbf{Fin}\upharpoonright\left(\omega\setminus\left\{0\right\}\right)\times\omega\right)$$
is weakly Ramsey. Then $\left(\emptyset\otimes \mathbf{Fin}\right)\oplus\mathcal{ED}$ is as needed, since $\mathcal{ED}$ is not locally selective (and therefore cannot be weakly selective).

In the next example we construct a locally selective ideal which is not weakly Ramsey. Later in this paper we will show that the ideal $\mathcal{WR}$ is another example of such an ideal (cf. Corollary \ref{WR}).

\begin{exa}
\label{example}
Let $\mathcal{ED}_\uparrow$ be the ideal on $\omega\times\omega$ generated by vertical lines (called \emph{generators of the first type}) and graphs of nondecreasing functions from $\omega$ to $\omega$ (called \emph{generators of the second type}). In other words $\mathcal{ED}_\uparrow$ is generated by homogeneous subsets of the coloring $\chi_\uparrow\colon \left[\omega\times\omega\right]^2\to 2$ given by:
$$\chi_\uparrow\left(\left\{\left(i,j\right),\left(k,l\right)\right\}\right)=\left\{\begin{array}{ll}
0 & \mbox{\boldmath{if }} i<k \mbox{\boldmath{ and }} j\leq l\\
1 & \mbox{\boldmath{if not}}\\
\end{array}\right.$$
for all $\left(i,j\right)$ below $\left(k,l\right)$ in the lexicographical order.\\
To show that $\mathcal{ED}_\uparrow$ is not weakly Ramsey, it suffices to observe that
$$\left\{b\in \omega\times\omega:\chi_\uparrow\left(\left\{a,b\right\}\right)=1\right\}\in\mathcal{ED}_\uparrow$$
for each $a\in\omega\times\omega$, since this set is covered by finitely many vertical lines and graphs of constant (so nondecreasing) functions. Hence, $\mathcal{ED}_\uparrow$ does not satisfy condition $(2)$ of Proposition \ref{EquivalentConditions}.\\
On the other hand, $\mathcal{ED}_\uparrow$ is locally selective. Indeed, let $(X_n)_{n\in\omega}\subset\mathcal{ED}_\uparrow$ be a partition of $\omega\times\omega$. We will find a selector $S$ of that partition not belonging to $\mathcal{ED}_\uparrow$. Observe that each $X_n$ is infinite only on finitely many vertical lines. There are two possible cases:\\
{\bf Case 1.} There are infinitely many vertical lines with infinite intersection with some $X_n$. In this case, there is an infinite set $T\subset\omega$ such that for each its element $t$ there is $k(t)\in\omega$ with $X_{k(t)}\cap\left(\left\{t\right\}\times\omega\right)$ infinite. By shrinking the set $T$ we can additionally assume that $k(t)\neq k(t')$ for $t,t'\in T$, $t\neq t'$. Enumerate $T=\left\{t_0,t_1,\ldots\right\}$ in such a way that $(t_n)_{n\in\omega}$ is a concatenation of finite decreasing sequences of increasingly larger lengths and such that all elements of the next finite decreasing sequence are greater than all elements of each previous finite decreasing sequence, i.e., for instance
$$t_0<t_2<t_1<t_5<t_4<t_3<t_9<t_8<t_7<t_6<t_{10}<\ldots$$
Inductively pick an increasing sequence $\left(m_i\right)_{i\in\omega}\subset\omega$. At the end $S$ will consist of all $\left(t_i,m_i\right)$. Let $m_0$ be any point with $\left(t_0,m_0\right)\in X_{k(t_0)}$. Suppose that $m_j$, for $j<i$, are constructed. Let $m_i$ be any point with $\left(t_i,m_i\right)\in X_{k(t_i)}$ such that $m_i>m_{i-1}$. Define $S=\left\{\left(t_i,m_i\right):i\in\omega\right\}$. Then $S$ is contained in some selector of the partition $(X_n)_{n\in\omega}$. On the other hand, suppose that $S$ is in $\mathcal{ED}_\uparrow$. Then there is $k$ such that $S$ is covered by $k$ many generators of the second type of the ideal $\mathcal{ED}_\uparrow$, i.e., $k$ many graphs of nondecreasing functions. Then at least two points of $\left(t_p,m_p\right),\ldots,\left(t_{p+k+1},m_{p+k+1}\right)$ are in one of them, where $p=1+2+\ldots +k$. However, it is impossible, since $t_{p+j}<t_{p+i}$ and $m_{p+j}>m_{p+i}$ for $i<j\leq k+1$. Hence, $S$ is not in $\mathcal{ED}_\uparrow$.\\
{\bf Case 2.} There are infinitely many vertical lines intersecting infinitely many $X_n$'s. In this case, there is an infinite set $T\subset\omega$ such that for each $t\in T$ the set 
$$\left\{i:X_i\cap\left(\left\{t\right\}\times\omega\right)\neq\emptyset\right\}$$
is infinite. Enumerate $T=\left\{t_0,t_1,\ldots\right\}$ in the same way as in Case $1$ and inductively pick an increasing sequence $\left(m_i\right)_{i\in\omega}$. Let $m_0$ be any point. Suppose that $m_j$, for $j<i$, are constructed. Let $m_i$ be any point with 
$$\left(t_i,m_i\right)\notin\bigcup\left\{X_k:\exists_{j<i} \left(t_j,m_j\right)\in X_k\right\}$$
and such that $m_i>m_{i-1}$. Then the set $S=\left\{\left(t_i,m_i\right):i\in\omega\right\}$ is a subset of some selector of the partition $(X_n)_{n\in\omega}$ and does not belong to $\mathcal{ED}_\uparrow$ for the same reasons as in Case 1.\\
Hence, $\mathcal{ED}_\uparrow$ is locally selective.
\end{exa}

The ideal $\mathcal{ED}_\uparrow$ from Example \ref{example} is $\bf{\Sigma^{0}_{2}}$. Indeed, the submeasure $\phi$ on $\omega\times\omega$ defined by
$$\begin{array}{rcl}
\phi(A) & = & \inf\left\{\left|\mathcal{C}\right|:A\subset\bigcup\mathcal{C}\textrm{ and each }C\in\mathcal{C}\textrm{ is either a generator}\right.\\ 
& & \left.\textrm{of the first or of the second type of the ideal }\mathcal{ED}_\uparrow\right\}\end{array}$$
is lower semicontinuous and $\mathcal{ED}_\uparrow=\mathbf{Fin}(\phi)$.

\begin{rem}
There is an ideal on $\omega$ isomorphic to $\mathcal{ED}_\uparrow$ which is not Mon.
\end{rem}

\begin{proof}
Denote $U=\left\{\left(i,j\right):j\geq i\right\}$ and $L=\left\{\left(i,j\right):j< i\right\}$. Define inductively a bijection $\pi\colon\omega\to\omega\times\omega$. In the $n$-th inductive step we define $\pi$ on numbers from $2(1+2+\ldots +n)$ to $2(1+\ldots+(n+1))-1$. Let $\pi(0)=(0,0)$ and $\pi(1)=(1,0)$. Suppose that $\pi(i)$ for $i<p=2(1+2+\ldots +n)$ are defined and let $\pi(p+2k)=(k,n)$ and $\pi(p+2k+1)=(n+1,n-k)$ for $k=0,\ldots,n$. Notice that $\pi[\{2k:k\in\omega\}]=U$ and $\pi[\{2k+1:k\in\omega\}]=L$.\\
Consider a sequence $\left(x_n\right)_{n\in\omega}$ such that for each $k\in\omega$ both $\left(x_n\right)\upharpoonright\pi[U\cap(\{k\}\times\omega)]$ and $\left(x_n\right)\upharpoonright\pi[L\cap(\omega\times\{k\})]$ are decreasing sequences valued in the open interval $(k,k+1)\subset\mathbb{R}$. Suppose that $M\subset\omega$ is such that the sequence $\left(x_n\right)\upharpoonright M$ is monotone. If it is nonincreasing, then $\pi[M]\cap U$ is covered by finitely many vertical lines and $\pi[M]\cap L$ is covered by finitely many horizontal lines (which are graphs of constant, so nondecreasing functions). Indeed, let $r\in\mathbb{R}$ be the smallest element of the sequence $\left(x_n\right)\upharpoonright M$. There is such $k$ that $r\in (k-1,k)$. Then $\pi[M]\cap U\subset k\times\omega$ and $\pi[M]\cap L\subset\omega\times k$. Hence, $\pi[M]$ is in $\mathcal{ED}_\uparrow$. On the other hand, if $\left(x_n\right)\upharpoonright M$ is nondecreasing, then both $\pi[M]\cap U$ and $\pi[M]\cap L$ are graphs of nondecreasing functions from $\omega$ to $\omega$. Indeed, let $n,m\in M\cap\pi^{-1}[U]$ be such that $n<m$ and $x_n\leq x_m$. Assume that $n=2(1+\ldots+(j+1))+2i$ and $m=2(1+\ldots+(l+1))+2k$ for some $i,j,k,l$ such that $i\leq j$ and $k\leq l$. Since $x_n\leq x_m$ then by the definition of the sequence $\left(x_n\right)_{n\in\omega}$ we have that $i<k$. We must show that $j\leq l$. But if $j>l$ then $n>m$. Therefore the set $\pi[M]\cap U$ is covered by a graph of a nondecreasing function. In the case of the set $\pi[M]\cap L$ the proof is similar. Hence, also in this case, $\pi[M]$ is in $\mathcal{ED}_\uparrow$.
\end{proof}

Next Proposition follows from Theorem \ref{MainTheorem}, however here we attach a direct proof.

\begin{prop}
\label{zawieranie}
Ideal $\mathcal{ED}_\uparrow$ contains an isomorphic copy of the ideal $\mathcal{WR}$ (so $\mathcal{WR}\sqsubseteq\mathcal{ED}_\uparrow$).
\end{prop}

\begin{proof}
We define a function $\pi\colon\omega\times\omega\to\omega\times\omega$ by
$\pi\left((i,2j)\right)=(i,i+j)$ and $\pi\left((i,2j+1)\right)=(i+j+1,i)$ for $i,j\in\omega$. Then $\pi[\bigcup_{n\in\omega}\omega\times\{2n\}]\subset\{(i,j):i\leq j\}$ and $\pi[\bigcup_{n\in\omega}\omega\times\{2n+1\}]\subset\{(i,j):i>j\}$. Observe that $\pi$ is $1-1$. We will show that $\pi[\omega\times\omega]=\omega\times\omega$ (so $\pi$ is onto). Let $(n,m)\in\omega\times\omega$. If $n\leq m$, then $(n,m)=\pi((n,2(m-n)))$ and if $n>m$, then $(n,m)=\pi((m,2(n-m-1)+1))$. \\
To conclude the proof we will show that for each generator $A$ of the ideal $\mathcal{WR}$ the set $\pi[A]$ belongs to $\mathcal{ED}_\uparrow$. Assume first that $A$ is a generator of the first type of the ideal $\mathcal{WR}$, i.e., $A=\{i\}\times\omega$ for some $i$. Then $\pi[A]$ is covered by the set $\{i\}\times\omega$ and by the graph of a constant (so nondecreasing) function equal to $i$. Therefore $\pi[A]$ belongs to $\mathcal{ED}_\uparrow$. Assume now that $A$ is a generator of the second type of the ideal $\mathcal{WR}$. Observe that
$$\pi[A]=\pi[A\cap\bigcup_{n\in\omega}\omega\times\{2n\}]\cup\pi[A\cap\bigcup_{n\in\omega}\omega\times\{2n+1\}].$$
We will show that both sets in the above sum are covered by graphs of nondecreasing functions. Let $(i,2j),(k,2l)$ belonging to $A\cap\bigcup_{n\in\omega}\omega\times\{2n\}$ be such that $k>i+2j$. Then $\pi\left((i,2j)\right)=(i,i+j)$ and $\pi\left((k,2l)\right)=(k,k+l)$. Since $k>i+2j$, then $k>i$ and $k+l>i+j$, hence $\pi[A\cap\bigcup_{n\in\omega}\omega\times\{2n\}]$ is covered by a graph of a nondecreasing function. Similarly, if $(i,2j+1),(k,2l+1)$ belonging to $A\cap\bigcup_{n\in\omega}\omega\times\{2n+1\}$ are such that $k>i+2j+1$, then $\pi\left((i,2j+1)\right)=(i+j+1,i)$ and $\pi\left((k,2l+1)\right)=(k+l+1,k)$. Moreover $k+l+1>i+j+1$ and $k>i$. Therefore also in this case $\pi[A\cap\bigcup_{n\in\omega}\omega\times\{2n+1\}]$ is covered by a graph of a nondecreasing function.
\end{proof}

\begin{cor}
\label{WR}
The ideal $\mathcal{WR}$ is another example of locally selective ideal which is not weakly Ramsey. 
\end{cor}

\begin{proof}
Indeed, if $\mathcal{WR}$ would not be locally selective, then it would contain an isomorphic copy of $\mathcal{ED}$. However, by Proposition \ref{zawieranie} we have that $\mathcal{WR}\sqsubseteq\mathcal{ED}_\uparrow$ and $\mathcal{ED}_\uparrow$ does not contain an isomorphic copy of $\mathcal{ED}$. A contradiction. On the other hand, $\mathcal{WR}$ is not weakly Ramsey by Corollary \ref{w.R.}.
\end{proof}

By Proposition \ref{zawieranie} $\mathcal{WR}\sqsubseteq\mathcal{ED}_\uparrow$. Next result shows that $\mathcal{WR}$ and $\mathcal{ED}_\uparrow$ are different ideals. Namely, $\mathcal{WR}$ and $\mathcal{ED}_\uparrow$ are not $\sqsubseteq$-equivalent. It shows us that the critical ideal for weak Ramseyness cannot be simplified. Also we can conclude that there are at least two isomorphic types of locally selective ideals which are not weakly Ramsey.

\begin{prop}
Ideals $\mathcal{WR}$ and $\mathcal{ED}_\uparrow$ are not $\sqsubseteq$-equivalent.
\end{prop}

\begin{proof}
We will show that $\mathcal{WR}$ does not contain an isomorphic copy of $\mathcal{ED}_\uparrow$. Suppose otherwise and denote by $X_n$ the preimage of the $n$-th vertical line under the bijection $\sigma:\omega\times\omega\rightarrow\omega\times\omega$ witnessing that $\mathcal{ED}_\uparrow\sqsubseteq\mathcal{WR}$. The proof will follow the same scheme as in Example \ref{example}. Observe that each $X_n$ has infinite intersection only with finitely many vertical lines. We will construct a set $S$ not belonging to $\mathcal{WR}$ and such that $\sigma[S]$ is covered by a graph of a nondecreasing function (so $\sigma[S]$ is in $\mathcal{ED}_\uparrow$). There are two possible cases:\\
{\bf Case 1.} There are infinitely many vertical lines, on which some $X_n$ is infinite. In this case, there is an infinite set $T=\left\{t_0<t_1<\ldots\right\}$ such that for each $n$ there is $k(t_n)$ with $X_{k(t_n)}\cap\left(\left\{t_n\right\}\times\omega\right)$ infinite. We can assume that $k(t_n)<k(t_m)$ for $n<m$ by picking an increasing subsequence of the sequence $\left(k(t_n)\right)_{n\in\omega}$. Inductively pick points $a_i\in\omega\times\omega$ for $i\in\omega$. At the end $S$ will consist of all $a_i$. Let $a_0$ be any point in $X_{k(t_0)}\cap\left(\left\{t_0\right\}\times\omega\right)$. Suppose that $a_j$, for $j<i$, are constructed. Let $a_i$ be any point in $X_{k(t_i)}\cap\left(\left\{t_i\right\}\times\omega\right)$ such that $\sum a_i>t_{2i}$ and $$\textrm{proj}_2\left(\sigma\left(a_{i}\right)\right)>\textrm{proj}_2\left(\sigma\left(a_{i-1}\right)\right).$$ 
We also have
$$\textrm{proj}_1\left(\sigma\left(a_{i}\right)\right)=k(t_i)>k(t_{i-1})=\textrm{proj}_1\left(\sigma\left(a_{i-1}\right)\right).$$ 
Define $S=\left\{a_i:i\in\omega\right\}$. Then $\sigma[S]$ is in $\mathcal{ED}_\uparrow$. On the other hand, suppose that $S$ is in $\mathcal{WR}$. Then there is $k$ such that $S$ is covered by $k$ many generators of the second type of the ideal $\mathcal{WR}$. However it is impossible by Remark \ref{uwaga} applied to points $a_p,\ldots,a_{p+k+1}$, where $p=1+2+\ldots +k$, since $\sum a_{p+i}>t_{2p}>t_{p+k+1}$ for $i\leq k+1$. \\
{\bf Case 2.} There are infinitely many vertical lines intersecting infinitely many $X_n$'s. In this case, there is an infinite set $T=\left\{t_0<t_1<\ldots\right\}$ such that for each $n$ the set 
$$\left\{i:X_i\cap(\left\{t_n\right\}\times\omega)\neq\emptyset\right\}$$
is infinite. Inductively pick points $a_i$, for $i\in\omega$, as follows. Let $a_0$ be any point in $\left\{t_0\right\}\times\omega$. Suppose that $a_j$, for $j<i$, are constructed and let $k$ be such that $a_{i-1}\in X_k$. Let $$a_i\in\left(\left\{t_{i}\right\}\times\omega\right)\cap\bigcup_{n>k}X_n,$$
be such that $\sum a_i>t_{2i}$ and $$\textrm{proj}_2\left(\sigma\left(a_{i}\right)\right)>\textrm{proj}_2\left(\sigma\left(a_{i-1}\right)\right).$$
Define $S=\left\{a_i:i\in\omega\right\}$. Then $\sigma[S]$ is in $\mathcal{ED}_\uparrow$, however $S$ is not in $\mathcal{WR}$ by Remark \ref{uwaga} for the same reasons as in Case $1$.\\
A contradiction. Hence, $\mathcal{WR}$ does not contain an isomorphic copy of $\mathcal{ED}_\uparrow$.
\end{proof}

\section{Characterization of weakly Ramsey ideals}

In this Section we write $(i,j)\sqsubset (k,l)$ ($(i,j)\sqsubseteq (k,l)$) for $(i,j),(k,l)\in\omega\times\omega$ if $i<k$ ($i\leq k$). Similarly, we write $(i,j)\sqsupset (k,l)$ ($(i,j)\sqsupseteq (k,l)$) if $i>k$ ($i\geq k$). 

\begin{defin}
\label{QC-pomocnicze}
For any function $\pi\colon\omega\times\omega\to\omega$ with
\begin{itemize}
	\item[(a)] $\pi[\omega\times\omega]=\omega$, i.e., $\pi$ is onto,
	\item[(b)] $\pi$ is finite-to-one,
\end{itemize}
let $\mathcal{WR}^\pi$ be an ideal on $\omega\times\omega$ generated by vertical lines (which we call \emph{generators of the first type of the ideal $\mathcal{WR}^\pi$}) and sets $G=\left\{g_0\sqsubset g_1\sqsubset\ldots\right\}$, such that $\pi(g_i)<\pi(g_j)$ and $g_j\sqsupseteq (\pi(g_i),0)$ for all $i<j$ (which we call \emph{generators of the second type of the ideal $\mathcal{WR}^\pi$}).
\end{defin}

\begin{rem}
\label{QC-uwaga}
Notice that $\mathcal{WR}$ is of the form $\mathcal{WR}^\pi$. Indeed, consider the function $\hat{\pi}\colon\omega\times\omega\to\omega$ given by $\hat{\pi}((i,j))=i+j+1$. Let $\pi\colon\omega\times\omega\to\omega$ be such that $\pi((i,j))=\hat{\pi}((i,j))$ for $(i,j)\neq (0,1)$ and $\pi((0,1))=0$. It is easy to see that $\pi$ satisfies conditions $(a)$ and $(b)$ from Definition \ref{QC-pomocnicze} and $\mathcal{WR}^{\pi}$ is equal to the ideal generated by vertical lines and sets $G=\left\{g_0\sqsubset g_1\sqsubset\ldots\right\}$, such that $\hat{\pi}(g_i)<\hat{\pi}(g_j)$ and $g_j\sqsupseteq (\hat{\pi}(g_i),0)$ for all $i<j$. We will show that $\mathcal{WR}=\mathcal{WR}^{\pi}$. Vertical lines are generators of both $\mathcal{WR}$ and $\mathcal{WR}^{\pi}$. If $G$ is a generator of the second type of the ideal $\mathcal{WR}$, then for any $(i,j),(k,l)\in G$ with $i+j<k$ we have $(i,j)\sqsubset (k,l)$ and $(i+j+1,0)\sqsubseteq (k,l)$. Moreover
$$\hat{\pi}((i,j))=i+j+1\leq k<k+l+1=\hat{\pi}((k,l)).$$
Therefore $G\in\mathcal{WR}^{\pi}$. On the other hand, let $G=\left\{g_0\sqsubset g_1\sqsubset\ldots\right\}\in\mathcal{WR}^{\pi}$ be such that $\hat{\pi}(g_i)<\hat{\pi}(g_j)$ and $(\hat{\pi}(g_i),0)\sqsubseteq g_j$ for all $i<j$. Then for any $(i,j),(k,l)\in G$ such that $i<k$, (i.e., $(i,j)\sqsubset(k,l)$) we have $(i+j+1,0)=(\hat{\pi}((i,j)),0)\sqsubseteq (k,l)$, so in particular $k>i+j$. Hence $G$ is a generator of the second type of the ideal $\mathcal{WR}$.
\end{rem}

\begin{lem}
\label{dense}
$\mathcal{WR}$ and $\mathcal{WR}^\pi$ are dense ideals.
\end{lem}

\begin{proof}
Let $\pi:\omega\times\omega\rightarrow\omega$ be any function satisfying conditions $(a)$ and $(b)$ from Definition \ref{QC-pomocnicze}. Take $A\notin\mathcal{WR}^\pi$. If there is $i$ such that $A$ has infinite intersection with $\left\{i\right\}\times\omega$, then define $B=A\cap(\left\{i\right\}\times\omega)$. If $A$ has finite intersection with every vertical line, then $A$ intersects infinitely many such lines and we construct an infinite subset of $A$ belonging to the ideal. Take any $x_0\in A$. If $x_0,\ldots,x_k$ are constructed, then let 
$$x_{k+1}\in A\cap\left\{x:x_k\sqsubset x\wedge x\sqsupseteq (\pi(x_k),0)\wedge\pi(x_k)<\pi(x)\right\}.$$
Let $B=\left\{x_0,x_1,\ldots\right\}$. In both cases $B\subset A$ is infinite and $B\in\mathcal{WR}^\pi$.\\
The proof for $\mathcal{WR}$ follows from the above by Remark \ref{QC-uwaga}.
\end{proof}

The following lemma is crucial for the proof of Theorem \ref{MainTheorem}.

\begin{lem}
\label{lemat}
All ideals of the form $\mathcal{WR}^\pi$ are $\sqsubseteq$-equivalent. Moreover, they are $\sqsubseteq$-equivalent to the ideal $\mathcal{WR}$.
\end{lem}

\begin{proof}
By Remark \ref{QC-uwaga} $\sqsubseteq$-equivalence of $\mathcal{WR}$ and the ideals $\mathcal{WR}^\pi$ is a consequence of $\sqsubseteq$-equivalence of the ideals $\mathcal{WR}^\pi$. We will show that for any functions $\pi,\pi_0:\omega\times\omega\rightarrow\omega$ satisfying conditions $(a)$ and $(b)$ from Definition \ref{QC-pomocnicze} there is a $1-1$ function $\sigma:\omega\times\omega\rightarrow\omega\times\omega$ such that $\sigma^{-1}[A]\in\mathcal{WR}^{\pi_0}$ for all $A\in\mathcal{WR}^\pi$.\\
We can assume that $\pi^{-1}\left[\left\{0\right\}\right]\cap\left(\left\{0\right\}\times\omega\right)\neq\emptyset$. Indeed, otherwise consider the function $\pi'$ such that $\pi'(0,0)=0$, $\pi'(0,n+1)=\pi(0,n)$ for $n\in\omega$ and $\pi'(a)=\pi(a)$ for $a\in(\omega\setminus\{0\})\times\omega$. Then it is easy to see that $\mathcal{WR}^{\pi'}=\mathcal{WR}^\pi$.\\
We define $\sigma$ inductively by picking a partial $1-1$ function $\sigma'$ and a sequence $\left(\sigma_n\right)_{n\in\omega}$ of partial $1-1$ functions, where $\sigma'$ and all $\sigma_n$ are defined on pairwise disjoint subsets of $\omega\times\omega$ and have pairwise disjoint ranges. At the end $\sigma=\bigcup_{n\in\omega}\sigma_n\cup\sigma'$. Firstly we deal with the sequence $\left(\sigma_n\right)_{n\in\omega}$ ($\sigma'$ will be defined in the further part of the proof).\\
In order to define $\left(\sigma_n\right)_{n\in\omega}$ we need to introduce a sequence of finite sets $\left(A_n\right)_{n\in\omega}$. Those sets will play double role: the range of each $\sigma_n$ will be exactly $\left\{2n\right\}\times\omega\setminus A_n$ and the domain of each $\sigma_n$ will depend on $A_n$ 
in a more complicated way. The sequence is defined as follows: $A_0=\emptyset$ and
$$A_n=\left\{a:a\sqsubset(2n+1,0)\wedge\pi(a)\leq2n\right\}$$
for $n>0$. Observe that $\left(A_n\right)_{n\in\omega}$ is nondecreasing. Moreover, sets $A_n$ for $n>0$ are finite and nonempty since $\pi^{-1}\left[\left\{0\right\}\right]\cap\left(\left\{0\right\}\times\omega\right)\neq\emptyset$.\\
The inductive construction of partial $1-1$ functions $\sigma_n$ requires an auxiliary nondecreasing sequence $\left(m_n\right)_{n\in\omega}\subset\omega$. The value of $m_n$ will depend on all $\sigma_j$, for $j<n$, and $\sigma_n$ will depend on all $m_j$ for $j\leq n$. We start with $m_0=1$ and $\sigma_0\colon\left\{0\right\}\times\omega\to\left\{0\right\}\times\omega$ equal to identity. If $m_0,\ldots,m_{n-1}$ and $\sigma_0,\ldots,\sigma_{n-1}$ are constructed, then 
$$m_n=\max\pi_0\left[\bigcup_{j<n}\sigma_j^{-1}\left[A_n\cap\left(\left\{2j\right\}\times\omega\setminus A_j\right)\right]\right].$$
Observe that $m_n$ is defined correctly, since each $\sigma_j$, for $j<n$, is $1-1$ and the sets $A_n$, for $n>0$, are finite nonempty. Moreover, those sets constitute a nondecreasing sequence, hence $m_n\geq m_{n-1}$. Finally, let
$$\sigma_n\colon\left\{a:(m_{n-1},0)\sqsubseteq a\sqsubset (m_n,0)\wedge \pi_0(a)>m_n\right\}\to\left\{2n\right\}\times\omega\setminus A_n$$
be any bijection (if the domain of $\sigma_n$ is empty, i.e., $m_{n-1}=m_n$, then $\sigma_n=\emptyset$). Observe that $\sigma_n$ defined in this way is $1-1$ and has pairwise disjoint domain and range with each $\sigma_j$ for $j<n$. This ends the construction of $\left(\sigma_n\right)_{n\in\omega}$.\\
Now we deal with $\sigma'$ defined on $B=\bigcup_{n\in\omega}B_n$, where
$$B_n=\left\{a:(m_{n-1},0)\sqsubseteq a\sqsubset (m_n,0)\wedge \pi_0(a)\leq m_n\right\}$$ 
(note that $\bigcup_{n\in\omega}\sigma_n$ is defined on the complement of $B$). Observe that $B=\bigcup_{n\in\omega}B_n$ has finite intersection with every vertical line. Enumerate $B=\left\{b_0,b_1,\ldots\right\}$ in such a way that $\pi_0(b_0)\leq\pi_0(b_1)\leq\ldots$ and if $\pi_0(b_i)=\pi_0(b_{i+1})$ then $b_{i+1}\sqsubset b_i$. Define functions $f,h:\omega\rightarrow\omega$ by
$$f(n)=\left|\left\{j:b_j\sqsubset \left(\pi_0\left(b_n\right),0\right)\right\}\right|$$
and
$$h(n)=\left|\left\{j:b_j\sqsubset b_n\right\}\right|.$$
Now we can define $\sigma'$:
\begin{itemize}
	\item $\sigma'(b_0)$ is any element of $\left\{2h(0)+1\right\}\times\omega$ with $\pi(\sigma'(b_0))>2f(0)+1$.
	\item $\sigma'(b_{k+1})$ is any element of $\left\{2h(k+1)+1\right\}\times\omega$ with $\pi(\sigma'(b_{k+1}))>2f(k+1)+1$ and $\pi(\sigma'(b_{k+1}))>\pi(\sigma'(b_{k}))$.
\end{itemize}
It is easy to see that $\sigma=\bigcup_{n\in\omega}\sigma_n\cup\sigma'$ is a $1-1$ function defined on $\omega\times\omega$, since each partial function $\sigma_n$ and $\sigma'$ is $1-1$ and those functions have pairwise disjoint domains and ranges.\\
We will show that $\sigma$ is as needed, i.e., $\sigma^{-1}[A]\in\mathcal{WR}^{\pi_0}$ for all $A\in\mathcal{WR}^\pi$. Observe that preimages under $\sigma$ of even vertical lines are covered by finitely many vertical lines and preimages under $\sigma$ of odd vertical lines are finite. Therefore preimages under $\sigma$ of generators of the first type of the ideal $\mathcal{WR}^\pi$ are in $\mathcal{WR}^{\pi_0}$. Assume that $G$ is a generator of the second type of $\mathcal{WR}^\pi$. We have $$\sigma^{-1}[G]=\sigma^{-1}[G\cap\sigma[B]]\cup\sigma^{-1}[G\cap\left(\omega\times\omega\setminus \sigma[B]\right)],$$ so it suffices to check if $\sigma^{-1}[G\cap\sigma[B]]$ and $\sigma^{-1}[G\cap(\omega\times\omega\setminus \sigma[B])]$ are in $\mathcal{WR}^{\pi_0}$. We first deal with the second set. Assume that $\sigma^{-1}[G\cap(\omega\times\omega\setminus \sigma[B])]=\left\{g_0\sqsubset g_1\sqsubset\ldots\right\}$. We will show that this set is covered by two generators of the second type of the ideal $\mathcal{WR}^{\pi_0}$ -- one consisting of $g_i's$ with even indexes and second consisting of $g_i's$ with odd indexes. We must show that $\pi_0(g_{n+1})>\pi_0(g_n)$ and $g_{n+2}\sqsupseteq(\pi_0(g_n),0)$ for each $n$. By the construction of $\sigma$ we have $\sigma(g_i)\sqsubset \sigma(g_j)$. Hence, since $\sigma(g_i)'s$ constitute a generator of the second type $G$ of the ideal $\mathcal{WR}^\pi$, we have also $\pi(\sigma(g_i))<\pi(\sigma(g_j))$ and $(\pi(\sigma(g_i)),0)\sqsubseteq \sigma(g_j)$, for $i<j$.\\
Take $n\in\omega$ and suppose that $\sigma(g_{n+1})\in\left\{2k\right\}\times\omega\setminus A_k$. Then we have $\pi_0(g_{n+1})>m_k$, since only points satisfying this condition go on $\left\{2k\right\}\times\omega\setminus A_k$. By $\sigma(g_{n+1})\sqsupseteq(\pi(\sigma(g_n)),0)$ we have also $\pi(\sigma(g_n))\leq 2k$. Moreover $\sigma(g_n)\sqsubset(2k+1,0)$, by $g_n\sqsubset g_{n+1}$. Hence $\sigma(g_n)\in A_k$. By the definition of $m_k$ we get that $m_k\geq\pi_0(g_n)$. Therefore $\pi_0(g_{n+1})>\pi_0(g_n)$ since $\pi_0(g_{n+1})>m_k$. Observe also that
$$g_{n+2}\sqsupseteq (m_{k},0)\sqsupseteq (\pi_0(g_n),0),$$
since $\sigma(g_{n+2})\in\bigcup_{i>k}\left\{2i\right\}\times\omega\setminus A_i$ (because $\sigma(g_{n+1})\in\left\{2k\right\}\times\omega\setminus A_k$ and $g_{n+1}\sqsubset g_{n+2}$). Hence $\sigma^{-1}[G\cap(\omega\times\omega\setminus \sigma[B])]$ is covered by two generators of the second type of the ideal $\mathcal{WR}^{\pi_0}$.\\
Now we deal with the set $\sigma^{-1}[G\cap\sigma[B]]$. It is equal to $\left\{b_{n_0}, b_{n_1},\ldots\right\}$ for some increasing subsequence $\left(n_i\right)_{i\in\omega}$. We will show that $\sigma^{-1}[G\cap\sigma[B]]$ can be covered by one generator of the second type of the ideal $\mathcal{WR}^{\pi_0}$. Take $i,j\in\omega$ such that $i<j$. By the construction of the partial function $\sigma'$ we have $\pi(\sigma(b_{n_{j}}))>\pi(\sigma(b_{n_{i}}))$. Therefore we get $\sigma(b_{n_{i}})\sqsubset\sigma(b_{n_{j}})$, since $G$ is a generator of the second type of the ideal $\mathcal{WR}^{\pi}$. As $\sigma(b_{n_i})\in\left\{2h(n_i)+1\right\}\times\omega$ and $\sigma(b_{n_j})\in\left\{2h(n_j)+1\right\}\times\omega$, then $h(n_i)< h(n_j)$. Hence, by the definition of the function $h$ we have $b_{n_i}\sqsubset b_{n_j}$. Moreover, by the properties of the picked enumeration of the set $B$ we have $\pi_0(b_{n_i})\leq\pi_0(b_{n_j})$ (since $n_i<n_j$) and even $\pi_0(b_{n_i})<\pi_0(b_{n_j})$ (since $b_{n_i}\sqsubset b_{n_j}$).\\
We have $\sigma(b_{n_{j}})\in\left\{2h(n_j)+1\right\}\times\omega$ and
$$\sigma(b_{n_{j}})\sqsupseteq (\pi(\sigma(b_{n_{i}})),0)\sqsupseteq (2f(n_{i})+1,0).$$
Therefore $h(n_j)>f(n_i)$ and by the definition of the function $f$ we get that $(\pi_0(b_{n_i}),0)\sqsubseteq b_{n_j}$, which concludes the proof of the fact that $\sigma^{-1}[G\cap\sigma[B]]$ can be covered by one generator of the second type of the ideal $\mathcal{WR}^{\pi_0}$ and the proof of the entire Lemma.
\end{proof}

Now we are ready to prove Theorem \ref{MainTheorem}.

\begin{proof}[Proof of Theorem \ref{MainTheorem}]
Without loss of generality we can assume that $\mathcal{I}$ is an ideal on $\omega$.\\
{\bf(2) $\Rightarrow$ (3):} Obvious.\\
{\bf(3) $\Rightarrow$ (1):} Assume that $\mathcal{WR}\leq_K\mathcal{I}$. By Corollary \ref{w.R.} and Proposition \ref{EquivalentConditions} the ideal $\mathcal{WR}$ is not weakly Ramsey as witnessed by the coloring $\lambda\colon \left[\omega\times\omega\right]^2\to 2$ given by
$$\lambda\left(\left\{\left(i,j\right),\left(k,l\right)\right\}\right)=\left\{\begin{array}{ll}
0 & \mbox{\boldmath{, if }} k>i+j\\
1 & \mbox{\boldmath{, if }} k\leq i+j\\
\end{array}\right.$$
for all $\left(i,j\right)$ smaller than $\left(k,l\right)$ in the lexicographical order.\\
We will show that $\mathcal{I}$ is not weakly Ramsey. Suppose that $f:\omega\rightarrow\omega\times\omega$ witnesses that $\mathcal{WR}\leq_K\mathcal{I}$. Define a coloring $\chi\colon [\omega]^2\to 2$ by
$$\chi\left(\left\{n,m\right\}\right)=\left\{\begin{array}{ll}
\lambda\left(\left\{f(n),f(m)\right\}\right) & \mbox{\boldmath{, if }} f(n)\neq f(m)\\
1 & \mbox{\boldmath{, if }} f(n)=f(m)\\
\end{array}\right.$$
for $n,m\in\omega$ with $n\neq m$. We have 
$$\left\{m\in \omega:\chi\left(\left\{n,m\right\}\right)=1\right\}\in\mathcal{I}$$
for all $n\in\omega$ since 
$$\left\{b\in \omega\times\omega:\lambda\left(\left\{a,b\right\}\right)=1\right\}\in \mathcal{WR}$$ 
for each $a\in\omega\times\omega$ and
$$f^{-1}[\left\{b\in \omega\times\omega:\lambda\left(\left\{f(n),b\right\}\right)=1\right\}\cup \{f(n)\}]=\left\{m\in \omega:\chi\left(\left\{n,m\right\}\right)=1\right\}.$$
Suppose that $H\subset\omega$ is such that $\chi\upharpoonright[H]^2$ is constant. Then also $\lambda\upharpoonright[f[H]]^2$ is constant, so $f[H]$ is in $\mathcal{WR}$. Since $f$ witnesses that $\mathcal{WR}\leq_K\mathcal{I}$, we have $H\subset f^{-1}[f[H]]\in\mathcal{I}$. Hence, $\mathcal{I}$ is not weakly Ramsey.\\
{\bf(1) $\Rightarrow$ (2):} Suppose that $\mathcal{I}$ is not weakly Ramsey. Then by condition $4.$ from Proposition \ref{EquivalentConditions} there is a partition $(X_n)_{n\in\omega}\subset\mathcal{I}$ of $\omega$, such that $h[\omega]\in\mathcal{I}$ for all increasing functions $h:\omega\rightarrow\omega$, with $h(n+1)\in\bigcup_{i>h(n)}X_i$ for each $n\in\omega$.\\
Assume first that all $X_n$'s are infinite. We will find a bijection $\pi:\omega\times\omega\rightarrow\omega$ such that $\pi[A]\in\mathcal{I}$ for all $A\in\mathcal{WR}^\pi$. Then by Lemma \ref{lemat} we have $\mathcal{WR}\sqsubseteq\mathcal{I}$. Let $\pi:\omega\times\omega\rightarrow\omega$ be a bijection such that $\pi^{-1}[X_n]=\left\{n\right\}\times\omega$ for $n\in\omega$. Images of all vertical lines are in $\mathcal{I}$ and if $G=\left\{g_0\sqsubset g_1\sqsubset\ldots\right\}$ is such that $\pi(g_i)<\pi(g_j)$ and $g_j\sqsupseteq (\pi(g_i),0)$, for $i<j$, then $h_0,h_1:\omega\rightarrow\omega$ given by $h_0(n)=\pi(g_{2n})$ and $h_1(n)=\pi(g_{2n+1})$ are increasing. Moreover
$$h_0(n+1)=\pi(g_{2n+2})\in\bigcup_{i\geq\pi(g_{2n+1})}\pi[\left\{i\right\}\times\omega]\subset\bigcup_{i>\pi(g_{2n})}\pi[\left\{i\right\}\times\omega]=\bigcup_{i>h_0(n)}X_i,$$
since $(\pi(g_{2n+1}),0)\sqsubseteq g_{2n+2}$, so
$$g_{2n+2}\in\bigcup_{i\geq\pi(g_{2n+1})}(\left\{i\right\}\times\omega).$$
Similarly, $h_1(n+1)\in\bigcup_{i>h_1(n)}X_i$. Hence, $\pi[G]=h_0[\omega]\cup h_1[\omega]\in\mathcal{I}$.\\
Now we proceed to the general case. We will find a bijection $\pi:\omega\times\omega\rightarrow\omega$ such that $\mathcal{WR}^\pi\sqsubseteq\mathcal{I}$. Define $g:\omega\rightarrow\{2i:i\in\omega\}$ by $g(n)=2n$. Notice that $g^{-1}(n)=\frac{n}{2}$. Let $f:g[\omega]\rightarrow\omega\times\omega$ be a $1-1$ function such that $f\left[g\left[X_n\right]\right]$ is contained in $\left\{n\right\}\times\omega$ but not equal to it. Let also $\pi:\omega\times\omega\rightarrow\omega$ be a bijection such that $\pi^{-1}\upharpoonright g[\omega]=f$ and $\pi^{-1}\upharpoonright (\omega\setminus g[\omega])$ is any bijection between $\omega\setminus g[\omega]$ and $(\omega\times\omega)\setminus f[g[\omega]]$. By Lemma \ref{dense} to show that $\mathcal{WR}^\pi\sqsubseteq\mathcal{I}$ it suffices to find a $1-1$ function witnessing that $\mathcal{WR}^\pi\leq_{K}\mathcal{I}$. Define a $1-1$ function $\sigma:\omega\rightarrow\omega\times\omega$ by $\sigma(n)=f(g(n))$ (so $\sigma(n)=f(2n)$). We will show that $\sigma$ witnesses that $\mathcal{WR}^\pi\sqsubseteq\mathcal{I}$.\\
Firstly, observe that $\sigma^{-1}[\left\{n\right\}\times\omega]=X_n\in\mathcal{I}$. If $G\cap \sigma[\omega]=\left\{g_0\sqsubset g_1\sqsubset\ldots\right\}$ is such that $\pi(g_i)<\pi(g_j)$ and $g_j\sqsupseteq (\pi(g_i),0)$, for $i<j$, then define $h:\omega\rightarrow\omega$ by $h(n)=\sigma^{-1}(g_{n+1})$. Notice that $h(n)=g^{-1}\left(f^{-1}(g_{n+1})\right)=g^{-1}\left(\pi(g_{n+1})\right)$, since $f^{-1}(g_{n+1})\in g[\omega]$ and $\pi^{-1}\upharpoonright g[\omega]=f$. Therefore
$$h(n)=g^{-1}\left(\pi(g_{n+1})\right)<g^{-1}\left(\pi(g_{n+2})\right)=h(n+1).$$
Observe also that $0\notin h[\omega]$ since $h(0)=g^{-1}\left(\pi(g_{1})\right)>g^{-1}\left(\pi(g_{0})\right)\geq 0$. Finally, notice that
$$g_{n+2}\sqsupseteq (\pi(g_{n+1}),0)=(f^{-1}(g_{n+1}),0)\sqsupset (\frac{f^{-1}(g_{n+1})}{2},0)=(h(n),0).$$
Hence, $g_{n+2}\in\bigcup_{i>h(n)}\{i\}\times\omega$ and $h(n+1)=\sigma^{-1}(g_{n+2})\in\bigcup_{i>h(n)}X_i$ since $X_i=\sigma^{-1}[\left\{i\right\}\times\omega]$. Therefore, $\sigma^{-1}[G]=h[\omega]\cup\{\sigma^{-1}[\{g_0\}]\}\in\mathcal{I}$, which concludes the entire proof.
\end{proof}

\textbf{Acknowledgments.}
Results of this paper grew out of a question raised by Piotr Szuca during the seminar at the Institute of Mathematics of the University of Gdańsk. The author would like to express his gratitude to Piotr Szuca and Piotr Zakrzewski for some helpful discussions.

\end{document}